\documentclass[a4paper,12pt]{amsart}

\usepackage[paper=a4paper,textwidth=400pt,textheight=620pt,includeheadfoot,twoside=True]{geometry}

\usepackage[active]{srcltx}
\usepackage{amsmath,amssymb, amsthm,amsxtra}

\usepackage[mathscr]{eucal}

\usepackage{bm}
\usepackage[all]{xy}

\usepackage{aliascnt}

\theoremstyle{plain}
\newtheorem{thm}{Theorem}[section]
\newtheorem*{thm*}{Theorem}
\newaliascnt{prop}{thm}
\newaliascnt{cor}{thm}
\newaliascnt{lem}{thm}
\newaliascnt{claim}{thm}
\newaliascnt{defn}{thm}
\newaliascnt{ques}{thm}
\newaliascnt{conj}{thm}
\newaliascnt{fact}{thm}
\newaliascnt{rem}{thm}
\newaliascnt{ex}{thm}
\newtheorem{prop}[prop]{Proposition}
\newtheorem{cor}[cor]{Corollary}
\newtheorem{lem}[lem]{Lemma}

\newtheorem*{prop*}{Proposition}
\newtheorem*{cor*}{Corollary}
\newtheorem*{lem*}{Lemma}
\newtheorem*{claim*}{Claim}
\theoremstyle{definition}

\newtheorem*{defn*}{Definition}
\newtheorem*{ques*}{Question}
\newtheorem*{conj*}{Conjecture}

\newtheorem*{prob*}{Problem}

\newtheorem{rem}[rem]{Remark}
\newtheorem{ex}[ex]{Example}
\newtheorem*{fact*}{Fact}
\newtheorem*{rem*}{Remark}
\newtheorem*{ex*}{Example}
\aliascntresetthe{prop}
\aliascntresetthe{cor}
\aliascntresetthe{lem}
\aliascntresetthe{claim}
\aliascntresetthe{defn}
\aliascntresetthe{ques}
\aliascntresetthe{conj}
\aliascntresetthe{fact}
\aliascntresetthe{rem}
\aliascntresetthe{ex}
\usepackage[pointedenum]{paralist}
\setdefaultenum{(a)}{(i)}{(1)}{(A)}

\usepackage{varioref}
\labelformat{equation}{\textnormal{(#1)}}
\labelformat{enumi}{\textnormal{(#1)}}

\usepackage[a4paper]{hyperref}

\def\textsectionN~{\textsection{}}

\renewcommand\phi{\varphi}
\renewcommand\epsilon{\varepsilon}
\renewcommand\leq{\leqslant}
\renewcommand\geq{\geqslant}

\makeatletter
\newcommand{\set}{%
  \@ifstar{\@setstar}{\@set}%
}%
\newcommand{\@setstar}[2]{\{\, #1 \mid #2 \,\}}
\newcommand{\@set}[1]{\{\, #1 \,\}}
\newcommand{\Set}{%
  \@ifstar{\SetR}{\@Set}%
}%
\newcommand{\SetR}[2]{\left\{\, #1 \;\left\vert\; #2 \,\right.\right\}}
\newcommand{\@Set}[1]{\left\{\, #1 \,\right\}}
\makeatother
\newcommand{\abs}[1]{\lvert #1 \rvert}

\newcommand{\ZZ}{\mathbb{Z}}

\DeclareMathOperator{\codim}{codim}
\DeclareMathOperator{\rk}{rk}
\DeclareMathOperator{\PGL}{PGL}
\DeclareMathOperator{\Hom}{Hom}
\DeclareMathOperator{\Sing}{Sing}

\DeclareMathOperator{\Cone}{Cone}

\newcommand{\pr}{\mathrm{pr}}
\newcommand{\sO}{\ensuremath{\mathscr{O}}}
\newcommand{\A}{\mathbb{A}}
\newcommand{\Gr}{\mathbb{G}}
\newcommand{\TT}{\mathbb{T}}
\newcommand{\PP}{{\mathbb{P}}}
\newcommand{\PN}{\PP^N}
\newcommand{\Pv}[1][N]{(\PP^{#1})\spcheck}

\newcommand{\textgene}[1]{\ \ \text{#1}\ \,}

\newcommand{\textand}{\textgene{and}}

\newcommand{\sP}{\mathscr{P}}

\newcommand{\HH}{\mathscr{H}}
\newcommand{\Hv}{\mathscr{H}^v}
\newcommand{\RNk}{1 \leq k \leq r}

\newcommand{\prP}{\pr_2}
\newcommand{\prH}{\pr_1}

\title[Cohomological characterization of hyperquadrics]{Cohomological characterization of hyperquadrics
  of odd dimensions  in characteristic two}%
\hypersetup{%
  pdftitle={Cohomological characterization of hyperquadrics
    of odd dimensions  in characteristic two},%
  pdfauthor={Katsuhisa FURUKAWA}%
}

\email{katu@toki.waseda.jp}
\author[K.~Furukawa]{Katsuhisa~FURUKAWA}
\urladdr{\url{http://www.aoni.waseda.jp/katu/index.html}}
\address{
  Department of Mathematics,
  School of Fundamental Science and Engineering,
  Waseda~University,
  Ohkubo~3-4-1, Shinjuku, Tokyo, 169-8555, Japan
}
\date{February 2, 2014}

\begin{document}

\maketitle

\begin{abstract}
  We consider characterizations of projective varieties
  in terms of their tangents.
  S.~Mori established the characterization of projective spaces in arbitrary characteristic
  by ampleness of tangent bundles.
  J.~Wahl characterized
  projective spaces in characteristic zero by cohomological condition of tangent bundles;
  in addition, he remarked that a counter-example in characteristic two
  is constructed from odd-dimensional hyperquadrics $Q_{2n-1}$ with $n > 1$.
  This is caused by
  existence of a common point in $\PP^{2n}$
  which every embedded tangent space to 
  the quadric contains.
  In general, a projective variety in $\PN$ is said to be {strange}
  if its embedded tangent spaces admit such a common point in $\PN$.
  A non-linear smooth projective curve
  is strange if and only
  if it is a conic in characteristic two (E.~Lluis, P.~Samuel).
  S.~Kleiman and R.~Piene showed that
  a non-linear smooth hypersurface in $\PN$
  is strange
  if and only if it is a quadric of odd-dimension in characteristic two.
  In this paper,
  we investigate complete intersections,
  and prove that,
  a non-linear smooth complete intersection in $\PN$ is strange
  if and only if it is a quadric in $\PN$ of odd dimension in characteristic two;
  these conditions are also equivalent to
  non-vanishing of $0$-cohomology of $(-1)$-twist of the tangent bundle.
\end{abstract}

\section{Introduction}

S.~Mori \cite{Mori} established the characterization of projective spaces in characteristic $p \geq 0$
by ampleness of tangent bundles.
The work has motivated approaches via tangential properties to characterizing several projective varieties.
J.~Wahl \cite{Wahl} characterized projective spaces in $p = 0$
by cohomological condition of tangent bundles;
in addition, he remarked that a counter-example in $p=2$ is constructed from odd-dimensional hyperquadrics
$Q_{2n-1}$ with $n > 1$ \cite[p.~316]{Wahl}. This is caused by
existence of a common point $v \in \PP^{2n}$
which every embedded tangent space to the quadric
contains (\autoref{thm:cohom-chara}).
In general, a projective variety $X \subset \PN$ is said to be \emph{strange}
if its embedded tangent spaces admit such a common point $v \in \PN$
(see \autoref{sec:prel-defin-polyn}, for more details).

It is classically known that
a non-linear smooth projective curve $X \subset \PN$
is strange if and only if $X$ is a conic in $\PN$ in $p=2$
(E.~Lluis \cite{Lluis}, P.~Samuel \cite{Samuel}; see also \cite[IV, Theorem~3.9]{hartshorne}).
In higher-dimensions, 
S.~Kleiman and R.~Piene {\cite[Theorem~7]{KP}} focused on hypersurfaces, and showed that
a non-linear smooth hypersurface $X \subset \PN$ is strange if and only if
$X$ is a quadric in $\PN$ of odd dimension in $p=2$.

In this paper, we investigate whether a smooth complete intersection variety
other than quadrics can be strange,
and answer it negatively by examining
a parameter space of strange complete intersections.
In consequence, we have:

\begin{thm}\label{thm:sing-str}
  Let $X$ be a smooth projective variety which is
  a non-linear complete intersection in $\PN$, and let
  $\sO_X(1) := \sO_{\PN}(1)|_X$.
  Then the following are equivalent:
  \begin{enumerate}
  \item $X$ is strange in $\PN$,
  \item $X$ is a quadric in $\PN$ of odd dimension in $p=2$.
  \end{enumerate}
  Moreover, in the case where $\dim(X) \geq 2$, the conditions \textnormal{(a)} and \textnormal{(b)} are equivalent to
  \begin{enumerate}\setcounter{enumi}{2}
  \item $H^0(X, T_X(-1)) \neq 0$.
  \end{enumerate}

\end{thm}

This paper is organized as follows:
In \autoref{sec:prel-defin-polyn}, we study some basic properties of strangeness.
In \autoref{sec:section-cohomology-}, considering relation between strangeness and cohomology,
we show the equivalence ``(a) $ \Leftrightarrow $ (c)''
(for the exceptional case $\dim(X) = 1$, see \autoref{thm:rem-dim-1}).
In \autoref{sec:defin-polyn-strange}, we analyze defining polynomials of strange varieties.
We say that an $(N-r)$-dimensional projective variety in $\PN$ is an $(e^1, \dots, e^r)$-complete intersection
if it is scheme-theoretically equal to an intersection of $r$ hypersurfaces of degrees $e^1, \dots, e^r$,
where $e^1, \dots, e^r$ are $r$ positive integers. 
In \ref{eq:defHv} of \autoref{sec:space-compl-inters},
we define an irreducible parameter space $\Hv$
of $(e^1, \dots, e^r)$-complete intersections being strange for $v$.
In order to show the implication ``(a) $ \Rightarrow $ (b)'',
it is essential to consider
the case where all $e^k > 1$ and where $\Hv$ is \emph{not} equal to
the parameter space of quadrics of odd dimensions in $p=2$.
In \autoref{sec:strange-vari-with},
we construct
an $(e^1, \dots, e^r)$-complete intersection variety $X_0 \subset \PN$
having an isolated singular point $\alpha \neq v$.
In \autoref{sec:dimens-incid-vari},
we take
the incidence scheme $I \subset \Hv \times \PN$
parameterizing pairs of strange complete intersections and their singular points,
and in addition, take
an irreducible component $\Lambda \subset I$
whose subset parameterizes the orbit of the pair $(X_0, \alpha)$
under automorphisms of $\PN$ with fixed point $v$.
Calculating the dimension of $\Lambda$ and using the existence of $(X_0, \alpha)$, we show that the projection $\Lambda \rightarrow \Hv$ is surjective in the case.
This means that every $(e^1, \dots, e^r)$-complete intersection variety belonging to $\Hv$ is singular (\autoref{thm:sing-str2}), yielding \autoref{thm:sing-str}.

\section{Cohomology and defining polynomials}

\subsection{Preliminary}
\label{sec:prel-defin-polyn}

Let $X \subset \PN$ be a projective variety
over an algebraically closed field $K$ of characteristic $p \geq 0$.
We say that $X$ is \emph{strange for a point $v \in \PN$}
if $v \in \TT_xX$ for any smooth point $x \in X$,
where $\TT_xX \subset \PN$ is the embedded tangent space to $X$ at $x$.
We simply say that $X$ is \emph{strange} in $\PN$ if $X$ is strange for some point of $\PN$.

Let $(z_0: z_1: \dots: z_N)$ be the homogeneous coordinates on $\PN$.
We denote by $f_{z_j} := \partial f / \partial z_j$
for a homogeneous polynomial $f \in K[z_0, z_1, \dots, z_N]$.

\begin{prop}\label{thm:hypsurf-str}
  Let $X \subset \PN$ be a hypersurface defined by a homogeneous polynomial $f$,
  and let $v = (1:0:\dots:0) \in \PN$.
  Then $X$ is strange for $v$ if and only if $f_{z_0}$ is the zero polynomial.
\end{prop}
\begin{proof}
  Let $a \in X$ be a smooth point, where recall that
  $\TT_aX = (f_{z_0}(a) z_0 + \dots + f_{z_N}(a) z_N = 0) \subset \PN$.
  Then $v \in \TT_aX$
  if and only if $f_{z_0}(a) = 0$.
  Hence $X$ is strange for $v$ if and only if $f_{z_0}|_X = 0$.
  Here, the latter condition $f_{z_0}|_X = 0$
  means that
  $f_{z_0}$ is contained in the ideal $(f) \subset K[z_0, \dots, z_N]$,
  and then we have $f_{z_0} = 0$ because of $\deg (f_{z_0}) < \deg(f)$.
\end{proof}

Note that, in the case where $X \subset \PN$ is a degenerate subvariety,
i.e., $X$ is contained in an $m$-dimensional linear subvariety $L$ of $\PN$ with $m < N$,
$X$ is strange in $\PN$ if and only if $X$ is strange in $L \simeq \PP^m$.

\begin{ex}\label{thm:quad}
  Let $X \subset \PN$ be a smooth quadric, i.e.,
  a smooth projective variety of degree $2$ in $\PN$.
  Then $X$ is strange if and only if $\dim(X)$ is odd and $p=2$.
  The reason is as follows:
  It is sufficient to consider the case where $X$ is non-degenerate;
  thus we set $N = \dim(X) + 1$.
  Let $f$ be the defining equation of $X$.
  Choosing suitable coordinates $(z_0: z_1: \dots: z_N)$ on $\PP^N$,
  we can assume that
  \[
  f = 
  \begin{cases}
    z_0^2 + z_1z_2 + z_3z_4 + \dots + z_{N-1}z_N & \text{if $\dim(X)$ is odd},

    \\
    z_0z_1 + z_2z_3 + \dots + z_{N-1}z_N & \text{if $\dim(X)$ is even}.
  \end{cases}
  \]
  
  \begin{inparaenum}
  \item 
    Assume that $\dim(X)$ is odd and $p=2$. Then ${f}_{z_0} = 2z_0 = 0$;
    hence it follows from \autoref{thm:hypsurf-str} that
    $X$ is strange for $v = (1:0:\dots:0)$.

  \item 
    Assume that $\dim(X)$ is even or $p\neq2$.
    Let us consider
    the Gauss map $\gamma: X \rightarrow \Pv$ sending $a \mapsto \TT_aX$,
    where $\Pv = \Gr(N-1, \PN)$ is the space of hyperplanes of $\PN$.
    Indeed, we can describe $\gamma$ by
    \[
    a \mapsto (f_{z_0}(a): \dots: f_{z_N}(a)).
    \]
    Then, by assumption, $\gamma$ is isomorphic with $\gamma(\sO(1)) = \sO(1)$,
    and $\gamma(X)$ is also a smooth quadric hypersurface in $\Pv$.
    We denote by $u^* \subset \Pv$ the set of hyperplanes containing a point $u \in \PN$.
    Then $u^*$ is a hyperplane of $\Pv$.
    If $X$ is strange for some $u \in \PN$, then we have
    $\gamma(X) \subset u^*$, a contradiction.
    Hence $X$ is not strange.
  \end{inparaenum}
\end{ex}

We say that a projective variety $X \subset \PN$ is a \emph{cone} with vertex $v \in \PN$
if the line $\overline{xv}$ is contained in $X$
for any $x \in X$.
If $X$ is a cone with vertex $v$, then $X$ is strange for $v$.

\begin{rem}\label{thm:cone-strange}
  In $p=0$, if $X$ is strange for $v$, then
  $X$ is a cone with vertex $v$.
  The reason is as follows:
  Let $\pi_v: \PN\setminus\set{v} \rightarrow \PP^{N-1}$
  be the linear projection from $v$,
  and let $d_x\pi_v: T_x\PN \rightarrow T_{\pi_v(x)}\PP^{N-1}$
  be the linear map between Zariski tangent spaces.
  For any smooth point $x \in X$ with $x \neq v$,
  since $T_x\overline{xv} \subset T_xX \cap \ker (d_x\pi_v)$, we have
  $\dim (d_x\pi_v(T_xX)) = \dim(X) -1$.
  It follows from $p=0$ that $\dim(\pi_v(X\setminus\set{v})) = \dim(X)-1$;
  hence $X$ is a cone with vertex $v$.
\end{rem}

\begin{rem}\label{thm:str-var-basic-prop}
  Let $X \subset \PN$ be a projective variety being strange for a point $v$.
  Then we immediately have the following properties:

  \begin{inparaenum}
  \item \label{thm:str-var-cutH}
    If $X$ is smooth and $L \subset \PN$ is a hyperplane not containing $v$, then $X \cap L$ is smooth.
    This is because,
    for each point $x \in X \cap L$, it follows from $v \in \TT_xX$ that $\TT_xX \not\subset L$;
    hence $X \cap L$ is smooth at $x$ and $\TT_x(X \cap L) = \TT_x(X) \cap L$.

  \item\label{thm:str-var-proj}
    Let $\pi_z: \PN \setminus \set{z} \rightarrow \PP^{N-1}$
    be a linear projection from a point $z \in \PN$ with $z \neq v$.
    Then the image $Y \subset \PP^{N-1}$ of $X$ under $\pi_z$ is strange for $\pi_z(v)$.
    This is because, $\TT_{\pi_z(x)} Y$ contains $\pi_z(\TT_xX)$ for a general point $x \in X$.

  \end{inparaenum}
\end{rem}

\subsection{Section of $0$-cohomology of $(-1)$-twist of a tangent bundle}
\label{sec:section-cohomology-}

Let $X \subset \PN$ be a smooth quasi-projective variety, and let $\sO_X(1) := \sO_{\PN}(1)|_X$.
We identify $\PN$ with
$(H^0(\PN, \sO(1))\spcheck \setminus {0})/ (K \setminus {0})$,
the projectivization of
the dual vector space of  $H^0(\PN, \sO(1))$.
Let
\[
\hat v \subset H^0(\PN, \sO(1))\spcheck
\]
be the one-dimensional vector subspace
corresponding to $v \in \PN$.
Considering the Euler sequence
$0 \rightarrow \sO_{\PN} \rightarrow H^0(\PN, \sO(1))\spcheck \otimes \sO_{\PN}(1) \xrightarrow{\xi} T_{\PN} \rightarrow 0$,
we can define a composite homomorphism $s_v$ of bundles on $\PN$ by
\[
s_v:
\hat v \otimes \sO_{\PN}(1) \hookrightarrow H^0(\PN, \sO(1))\spcheck \otimes \sO_{\PN}(1) \xrightarrow{\xi} T_{\PN}.
\]

\begin{prop}\label{thm:section-cohom}
  Let $X \subset \PN$ be a smooth quasi-projective variety.
  Then $X$ is strange for $v$ if  and only if
  $s_v|_X$ factors through $T_X \subset T_{\PN}|_X$.
  Hence, in this case, $s_v|_X$
  gives a nonzero section of $H^0(X, T_{X} (-1)) \simeq \Hom_{\sO_X}(\sO_X(1), T_X)$.
\end{prop}

To show this, we consider
$\sP^1_{X} = \sP^1_{X}(\sO_X(1))$, the bundle of principal parts
of $\sO_X(1)$ of first order,
which gives an exact sequence
\begin{equation*}
  0 \rightarrow
  \Omega_{X}^1 \rightarrow
  \sP_X^1 \otimes \sO_X(-1)\rightarrow
  \sO_X \rightarrow
  0
\end{equation*}
(see \cite[Remark~(6.4)]{Piene1976}).
Taking the dual of this, we have the following commutative diagram with exact rows and columns:
\begin{equation}\label{eq:prin-P-X}
  \begin{split}
    \xymatrix{%
      && 0 \ar[d] & 0 \ar[d]
      \\
      0 \ar[r]& \sO_{X} \ar[r] \ar@{=}[d]& {\sP^1_{X}}\spcheck \otimes \sO_X(1) \ar[r] \ar[d]& T_{X} \ar[r] \ar[d]& 0
      \\
      0 \ar[r]& \sO_{X} \ar[r] & H^0(\PN, \sO(1))\spcheck \otimes \sO_{X}(1) \ar[r]^{\hspace{4em}\xi|_X} \ar[d] & T_{\PN}|_{X} \ar[r] \ar[d]& 0
      \\
      && N_{X/\PN} \ar@{=}[r] \ar[d] & N_{X/\PN} \ar[d]
      \\
      && 0 & 0 \makebox[0pt]{\, .}
    }%
  \end{split}
\end{equation}

\begin{rem}\label{thm:prin-P-corres-TTX}
  For a smooth point $x \in X$, the projectivization of
  ${\sP^1_X}\spcheck \otimes k(x) \subset H^0(\PN, \sO(1))\spcheck$ corresponds to the embedded tangent space $\TT_xX \subset \PN$. The reason is as follows.
  Let $r = \codim(X, \PN)$, and
  let $f^1, \dots, f^r$ be $r$ homogeneous polynomials
  which locally define $X$ around
  the smooth point $x$.
  Then the linear map
  $H^0(\PN, \sO(1))\spcheck \rightarrow N_{X/\PN} \otimes k(x)$
  induced from the middle column of the above diagram \ref{eq:prin-P-X}
  is represented by the matrix $[f^k_{z_i}(x)]_{1 \leq k \leq r, 0 \leq i \leq N}$,
  and hence its kernel is the zero set of $r$ linear polynomials
  $\sum_i f^1_{z_i}(x) \cdot z_i, \dots, \sum_i f^r_{z_i}(x) \cdot z_i$.
  This implies the assertion.
\end{rem}

From \autoref{thm:prin-P-corres-TTX},
we find that $\hat v \subset {\sP^1_X}\spcheck \otimes k(x)$
if and only if $v \in \TT_xX$.
In particular, the following holds:
\begin{lem}\label{thm:vOx-Px}
  The subbundle $\hat v \otimes \sO_X \subset H^0(\PN, \sO(1)) \otimes \sO_X$ is contained in $ {\sP^1_X}\spcheck$
  if and only if    $X$ is strange for $v$.
\end{lem}

\begin{proof}[Proof of \autoref{thm:section-cohom}]
  If $X$ is strange for $v$, then it follows from \autoref{thm:vOx-Px} that $s_v|_X$ is equal to the composite map,
  $\hat v \otimes \sO_X(1) \hookrightarrow
  {\sP^1_{X}}\spcheck \otimes \sO_X(1) \rightarrow T_{X}$.
  If $X$ is not strange for $v$, then $\hat v \not\subset {\sP^1_X}\spcheck \otimes k(x)$ for some $x$,
  and then the image of $s_v(x): \hat v \rightarrow T_x\PN$ is not contained in $T_xX$.
\end{proof}

\begin{rem}\label{thm:cohom-chara}
  If a smooth projective variety $X$ is strange and is not isomorphic to a projective space, then it follows from \autoref{thm:section-cohom} that
  $X$ gives a counter-example in $p > 0$ of the statement of Wahl's cohomological characterization of projective spaces.
  For example, smooth quadrics in $p=2$ whose dimensions are odd and $\geq 3$
  (\autoref{thm:quad}).
  (One dimensional quadrics, i.e., conics, are still isomorphic to $\PP^1$.)
\end{rem}

In addition, we can restate strangeness of tangents as a cohomological condition, as follows:
\begin{cor}\label{thm:cohom-prin-strange}
  Let $X \subset \PN$ be a smooth projective variety
  with $\sO_X(1) := \sO_{\PN}(1)|_{X}$, and
  denote by $\sP_X^1 = \sP_{X}^1(\sO_X(1))$ the bundle of principal parts of $\sO_X(1)$ of first order.
  Then $X$ is strange in $\PN$ if and only if
  $H^0(X, {\sP_X^1}\spcheck) \neq 0$.
\end{cor}
\begin{proof}
  The ``only if'' part follows immediately from \autoref{thm:vOx-Px}.
  To show the ``if'' part,
  we consider the case where there exists a nonzero section $s \in H^0(X, {\sP_X^1}\spcheck)$.
  By the middle column of \ref{eq:prin-P-X}, $s$ is regarded as a nonzero section of $H^0(\PN, \sO(1))\spcheck$,
  which gives a point $v \in \PN$ such that $K \cdot s = \hat v$.
  Then $X$ is strange for $v$.
\end{proof}

Now, let us consider the case where $X$ satisfies $H^1(X, \sO_X(-1)) = 0$.
\begin{prop}\label{thm:ci-cohom-cond}
  Let $X \subset \PN$ be a smooth projective variety,
  and assume that $H^1(X, \sO_X(-1)) = 0$.
  Then $X$ is strange in $\PN$ if and only if
  $H^0(X, T_X(-1)) \neq 0$.
\end{prop}
\begin{proof}
  The ``only if'' part was shown in \autoref{thm:section-cohom}.
  Let us consider the ``if'' part.
  By the first row of \ref{eq:prin-P-X},
  it follows from $H^1(X, \sO_X(-1)) = 0$ that
  \[
  H^0(X, {\sP_X^1}\spcheck) \rightarrow H^0(X, T_X(-1))
  \]
  is surjective.
  Hence the assertion follows from
  \autoref{thm:cohom-prin-strange}.
\end{proof}

\begin{lem}\label{thm:ci-O--1}
  Let $X = X^1 \cap \dots \cap X^r \subset \PN$ be a complete intersection of
  $r$ hypersurfaces $X^1, \dots, X^r$ of degrees $e^1, \dots, e^r$.
  Then $H^j(X, \sO_X(i)) = 0$ for $0 < j < N-r$ and all $i \in \ZZ$.
\end{lem}
\begin{proof}
  Since $H^j(\PN, \sO_{\PN}(i)) = 0$ for $0 < j < N$ and all $i$,
  it follows from the exact sequence
  $0 \rightarrow \sO_{\PN}(-e^1) \rightarrow \sO_{\PN} \rightarrow \sO_{X^1} \rightarrow 0$
  that $H^j(X^1, \sO_{X^1}(i)) = 0$ for $0 < j < N-1$ and all $i$.
  Similarly, 
  it follows from the exact sequence
  $0 \rightarrow \sO_{X^1}(-e^2) \rightarrow \sO_{X^1} \rightarrow \sO_{X^1 \cap X^2} \rightarrow 0$
  that $H^j (X^1 \cap X^2, \sO_{X^1 \cap X^2}(i)) = 0$ for $0 < j < N-2$ and all $i$.
  Inductively, we have
  $H^j (X, \sO_{X}(-i)) = 0$ for $0 < j < N-r$ and all $i$.
\end{proof}

\begin{rem}\label{thm:rem-dim-1}
  Let $X \subset \PN$ be a smooth complete intersection of dimension $\geq 2$.
  Then $X$ satisfies the statement of \autoref{thm:ci-cohom-cond}, because of \autoref{thm:ci-O--1}.
  Note that the case where $\dim (X) = 1$ is exceptional:
  Assume that $X$ is a smooth $(e^1, \dots, e^{N-1})$-complete intersection curve in $\PN$.
  Then $T_X = \sO_X(N+1-\sum_{1 \leq k \leq N-1} e^k)$, which implies that
  $H^0(X, T_X(-1)) \neq 0$ if and only if $X$ is a conic or line (in any $p \geq 0$).
  We recall that a smooth conic in $p \neq 2$ is not strange.
\end{rem}

\subsection{Defining polynomials of a strange complete intersection variety}
\label{sec:defin-polyn-strange}

Let $e^1, \dots, e^r$ be $r$ integers greater than $1$.
We recall that an $(N-r)$-dimensional projective variety $X \subset \PN$ is an $(e^1, \dots, e^r)$-complete intersection
if $X$ is scheme-theoretically equal to an intersection of $r$ hypersurfaces of degrees $e^1, \dots, e^r$,
i.e., a minimal set of generators of the defining homogeneous ideal $I_X \subset K[z_0, z_1, \dots, z_N]$ of $X$ consists of $r$ homogeneous polynomials of degrees $e^1, \dots, e^r$.
We generalize \autoref{thm:hypsurf-str}, as follows:

\begin{prop}\label{thm:fk_z0=0}
  Let $X \subset \PN$ be an $(e^1, \dots, e^r)$-complete intersection variety
  which is strange for a point $v = (1:0:\dots:0)$.
  Then $I_X$ is generated by $r$ homogeneous polynomials $f^1, \dots, f^r$ of degrees $e^1, \dots, e^r$
  such that $f^k_{z_0}$ is the zero polynomial for $\RNk$.
\end{prop}

To prove the above statement, we first show the following lemma:

\begin{lem}\label{thm:fk_z0=0-sub}
  Let $X \subset \PN$ be a projective variety which is strange for $v = (1:0:\dots:0)$.
  For a homogeneous polynomial $g \in I_X$ of degree $e$,
  there exists a homogeneous polynomial $\tilde g \in I_X$ of degree $e$,
  such that $\tilde g_{z_0} = 0$
  and the two ideals $(g,z_0), (\tilde g,z_0) \subset K[z_0, z_1, \dots, z_N]$
  coincide.
\end{lem}

\begin{proof}
  For a homogeneous polynomial $g \in I_X$, we have $g_{z_0} \in I_X$. The reason is as follows:
  There is nothing to prove if $g$ is the zero polynomial.
  Let $g$ be nonzero,
  and denote by $Y := (g=0) \subset \PN$, the hypersurface defined by $g$.
  For any $a \in X$,
  it follows that $v \in \TT_aX \subset \TT_aY$;
  then, since $v = (1:0:\dots:0)$ and $\TT_aY = (g_{z_0}(a) z_0 + \dots + g_{z_N}(a) z_N = 0)$,
  we have $g_{z_0}(a) = 0$.
  Hence $g_{z_0} \in I_X$.

  Applying the above argument inductively,
  we have $\partial^j g/ \partial z_0^j \in I_X$ for any $j > 0$.
  Now let
  \[
  \tilde g := g + \sum_{j=1}^{p-1} \frac{(-1)^j z_0^j}{j!} \cdot \frac{\partial^j g}{\partial z_0^j},
  \]
  which is contained in $I_X$ and satisfies that $(g,z_0) = (\tilde g,z_0)$.
  In addition, we have $\tilde g_{z_0} = 0$, because of
  \begin{multline*}
    \frac{\partial}{\partial z_0}
    \left( \sum_{j=1}^{p-1} \frac{(-1)^j z_0^j}{j!} \cdot \frac{\partial^j g}{\partial z_0^j} \right)
    =
    \sum_{j=1}^{p-1} \frac{(-1)^j z_0^{j-1}}{(j-1)!} \cdot \frac{\partial^j g}{\partial z_0^j}
    +
    \sum_{j=1}^{p-1} \frac{(-1)^j z_0^j}{j!} \cdot \frac{\partial^{j+1} g}{\partial z_0^{j+1}}
    \\
    = - \frac{\partial g}{\partial z_0}
    + \frac{(-1)^{p-1} z_0^{p-1}}{(p-1)!} \cdot \frac{\partial^{p} g}{\partial z_0^{p}}
    = - \frac{\partial g}{\partial z_0},
  \end{multline*}
  where $\partial^{p} g/ \partial z_0^{p} = 0$ since $p$ is the characteristic of the ground field.
\end{proof}

\begin{rem}
  The operation making $\tilde g$ with $\tilde g_{z_0} = 0$
  has appeared in an algorithm of derivation kernel computation 
  (see
  \cite[p.~27]{Essen};
  for positive characteristic,
  see \cite{Okuda}).
\end{rem}

\begin{proof}[Proof of \autoref{thm:fk_z0=0}]
  We can assume $N-r \geq 1$.
  Let $I_X$ be generated by homogeneous polynomials $f^1, \dots, f^r$ of degrees $e^1, \dots, e^r$.
  From \autoref{thm:fk_z0=0-sub}, for each $\RNk$,
  we have a homogeneous polynomial $\tilde f^k$ of degree $e^k$
  satisfying that
  $\tilde f^k \in I_X$,
  $\tilde f^k_{z_0} = 0$, and $(f^k,z_0) = (\tilde f^k,z_0)$.
  We set $\tilde X \subset \PN$ to be the zero set of
  $\tilde f^1, \dots, \tilde f^r$. Then we have $X \subset \tilde X$
  and $X \cap (z_0 = 0) = \tilde X \cap (z_0 = 0)$ for the hyperplane $(z_0 = 0) \subset \PN$.

  Let $V_1, \dots, V_m$ be the irreducible components of $\tilde X$,
  where $V_i$ is of dimension $\geq N-r \geq 1$ for each $1 \leq i \leq m$.
  Since $X$ is strange for $v$ and $v \notin (z_0 = 0)$,
  we have $X \not\subset (z_0 = 0)$; in particular, $\dim (X \cap (z_0 = 0)) = N-r-1$.
  Since $X \cap (z_0 = 0) = \bigcup_i (V_i \cap (z_0 = 0))$,
  we have $\dim(V_i \cap (z_0 = 0)) = N-r-1$,
  and hence $\dim(V_i) = N-r$.
  In particular, $X$ coincides with some $V_i$.
  Since $\deg(X) = \deg(\tilde X)$, we have $X = \tilde X$.
  %
\end{proof}

\begin{cor}\label{thm:e_k--p}
  Let $X \subset \PN$ be an $(e^1, \dots, e^r)$-complete intersection variety.
  Assume that $X$ is strange for a point $v$ and assume that $e^k < p$
  for any $k$ with $\RNk$. 
  Then $X$ is a cone with vertex $v$.
\end{cor}
\begin{proof}
  Changing coordinates, we can assume that $v = (1:0:\dots:0) \in \PN$.
  Then it follows from \autoref{thm:fk_z0=0} that
  $X$ is defined by homogeneous polynomials $f^1, \dots, f^r$ of degrees $e^1, \dots, e^r$ such that $f^k_{z_0} = 0$ for $\RNk$.
  If $f^k \notin K[z_1, \dots, z_N]$,
  then the inequality $e^k < p$ implies $f^k_{z_0} \neq 0$, a contradiction.
  Hence $f^k \in K[z_1, \dots, z_N]$ for any $k$, which means that
  $X$ is a cone with vertex $v$.
\end{proof}

\section{Parameter space of strange complete intersections}
\label{sec:space-compl-inters}

We fix some notation.
Let $\HH_e = \abs{\sO_{\PN}(e)}$ be the projectivization of $H^0(\PN, \sO(e))$,
whose general members parameterize hypersurfaces of $\PN$ of degree $e$.
For a point $v \in \PN$,
we denote by $\HH_e^v$ the subset of $\HH_e$
whose general members parameterize hypersurfaces being strange for $v$.
Changing homogeneous coordinates $(z_0: z_1: \dots: z_N)$ on $\PN$, we can assume
\[
v = (1: 0: \dots: 0).
\]
Then we have
$\HH_e^v = \set*{f \in \HH_e}{f_{z_0} = 0}$
due to \autoref{thm:hypsurf-str}.
In particular, $\HH_e^v$ is regarded as a linear subvariety of $\HH_e$,
since it is equal to the projectivization of the kernel of the linear map
$H^0(\PN, \sO(e)) \rightarrow H^0(\PN, \sO(e-1)): f \mapsto f_{z_0}$.

Let $e^1, \dots, e^r$ be $r$ integers greater than $1$.
We denote by
\begin{equation}\label{eq:defHH}
  \HH := \HH_{e^1} \times \dots \times \HH_{e^r},
\end{equation}
whose general members parameterize
$(N-r)$-dimensional $(e^1, \dots, e^r)$-complete intersections in $\PN$.
Let us consider
\begin{equation}\label{eq:defHv}
  \Hv := \HH_{e^1}^v \times \dots \times \HH_{e^r}^v,
\end{equation}
whose general members parameterize complete intersections
being strange for $v$. Indeed, we have:
\begin{equation}\label{eq:I-fz0-zero}
  \Hv = \set*{(f^1, \dots, f^r) \in \HH)}
  {f^1_{z_0} = \dots = f^r_{z_0} = 0}.
\end{equation}
Note that $\Hv$ is irreducible, since so is each $\Hv_e$.

\begin{rem}\label{thm:defpolyX-in-H}
  Let $X \subset \PN$ be an $(e^1, \dots, e^r)$-complete intersection variety
  being strange for $v$. Then, from \autoref{thm:fk_z0=0}, we can find
  homogeneous polynomials $(f^1, \dots, f^r) \in \Hv$
  whose zero set is equal to $X$.
\end{rem}

\subsection{Strange varieties with isolated singular points}
\label{sec:strange-vari-with}
We construct a member of $\Hv$ which defines a complete intersection variety $X_0$
having an isolated singular point $\alpha \neq v$.
Later, the pair $(X_0, \alpha)$ will play an essential role.

For a member $(f^1, \dots, f^r) \in \Hv$ and for $a \in \PN$, we set
\begin{equation}\label{eq:defn-Dfa}
  D ((f^k), a) = D ((f^k), a; \PN) :=
  \begin{bmatrix}
    f^1_{z_1} (a) & \cdots & f^1_{z_{N}} (a)
    \\
    \vdots && \vdots
    \\
    f^r_{z_1} (a) & \cdots & f^r_{z_{N}} (a)
  \end{bmatrix},
\end{equation}
where we need not consider $f^k_{z_0}$'s since these are zero as in \ref{eq:I-fz0-zero}.
The complete intersection variety $X \subset \PN$ defined by $(f^1, \dots, f^r)$
is singular at $a \in \PN$ if and only if
$f^1(a) = \dots = f^r(a) = 0$ and $\rk D((f^k), a) < r$.

\begin{prop}\label{thm:exist-barX}
  Assume that $\HH$ is \emph{not}
  equal to the parameter space of quadric hypersurfaces in $\PN$
  of odd dimensions in $p=2$.
  Assume that $p > 0$ and $e^k \geq p$ for some $k$.
  Then
  there exists a complete intersection variety $X_0 \subset \PN$ defined by a member of $\Hv$
  such that $\Sing X_0 \neq \set{v}$ and
  $0 < \#(\Sing X_0) < \infty$.
  In particular, there exists an isolated singular point $\alpha$ of $X_0$ such that
  $\alpha \neq v$.
\end{prop}

\begin{rem}[Hypersurfaces]\label{thm:exist-barX-hypsurf}
  Let us consider the case where $r=1$, i.e., $\HH$ is equal to $\HH_e$,
  the parameter space of hypersurfaces in $\PN$ of degree $e$.

  \begin{inparaenum}
  \item 
    Suppose that $\HH = \HH_2$ in $p = 2$, and suppose that $N$ is even (i.e., hypersurfaces are of odd dimensions).
    Then every member of $\Hv_2$ does \emph{not} satisfy
    the statement ``$\Sing X_0 \neq \set{v}$ and $0 < \#(\Sing X_0) < \infty$''.
    This is because, if a quadric $X \subset \PN$ is defined by a member of $\Hv$ and is singular at a point
    $\alpha \neq v$,
    then we have $\dim (\Sing X) \geq 1$, as follows:
    Let $\pi: \PN \setminus \set{\alpha} \rightarrow \PP^{N-1}$ be the linear projection from $\alpha$.
    Then $Y = \pi(X) \subset \PP^{N-1}$ is an $(N-2)$-dimensional quadric,
    and $X = \overline{\pi^{-1}(Y)}$ (it is a cone with vertex $\alpha$).
    Since $X$ is strange for $v$, $Y$ is strange for $\pi(v)$ as in \autoref{thm:str-var-basic-prop}\ref{thm:str-var-proj}.
    From \autoref{thm:quad}, $Y$ is not smooth. Hence, for $w \in \Sing(Y)$,
    the line $\overline{\pi^{-1}(w)}$ is contained in $\Sing(X)$.

  \item 
    Suppose that $\HH = \HH_2$ in $p=2$,
    and suppose that $N$ is odd.
    Then $\Hv_2$ satisfies the statement, as follows:
    We take $L \subset \PN$ to be a hyperplane containing $v$,
    and take $Y$ to be an $(N-2)$-dimensional smooth quadric in $L$ which is strange for $v$.
    For a point $\alpha \notin L$, we set
    $X_0 = \Cone_{\alpha}(Y)$, the cone over $Y$ with vertex $\alpha$.
    Then it follows that $\Sing(X_0) = \set{\alpha}$ with $\alpha \neq v$.

  \item
    For $e \geq 3$ and $e \geq p > 0$, we can construct a member of $\HH_e^v$
    satisfying the statement of \autoref{thm:exist-barX},
    as follows.
    Suppose that $p \mid e$. Then we set $X_0 \subset \PN$ to be the hypersurface defined by
    \[
    f = z_Nz_{N-1}^{e-1} + z_{N-1}z_{N-2}^{e-1} + \dots + z_2z_1^{e-1} + z_0^e.
    \]
    Then $X_0$ is strange for $v$ because of $f_{z_0} = 0$.
    On the other hand,
    $D(f, z) = \begin{bmatrix}
      f_{z_1}(z) & \dots & f_{z_N}(z)
    \end{bmatrix}
    $ is equal to
    \[
    \begin{bmatrix}
      (e-1)z_1^{e-2}z_2 & z_1^{e-1} + (e-1)z_2^{e-2}z_3& \dots
      & z_{N-2}^{e-1} + (e-1)z_Nz_{N-1}^{e-2} & z_{N-1}^{e-1}
    \end{bmatrix}.
    \]
    Then $\Sing (X_0) = \set{(0: \dots: 0: 1)}$ holds, as follows:
    We have ``$\supset$'' immediately. To show ``$\subset$'', we take $z \in \Sing(X_0)$.
    Since $D(f, z) = 0$,
    by the first polynomial from the right of the above description of $D(f,z)$,
    we have $z_{N-1} = 0$. By the second polynomial from the right,
    we have $z_{N-2} = 0$ because of $e-2 \geq 1$. Similarly, we have $z_1 = \dots = z_{N-1} = 0$.
    Since $f(z) = 0$, we have $z_0 = 0$.
    Therefore $z = (0: \dots: 0: 1)$.

    Suppose that $e > p$ and $p \nmid e$.
    Then we set $X_0 \subset \PN$ to be the hypersurface defined by
    \[
    f = z_0^pz_1^{e-p} + z_2^e + \dots + z_N^e.
    \]
    Then $X_0$ is strange for $v$. In addition, $D(f, z)$ is equal to
    \[
    \begin{bmatrix}
      (e-p) z_0^pz_1^{e-p-1} & e z_2^{e-1} & \dots & e z_N^{e-1}
    \end{bmatrix}.
    \]
    Thus $\Sing (X_0) = \set{(1:0:0:\dots:0), (0:1:0:\dots:0)}$ if $e > p+1$,
    and $\Sing (X_0) = \set{(0:1:0:\dots:0)}$ if $e = p+1$.
  \end{inparaenum}
\end{rem}

\begin{proof}[Proof of \autoref{thm:exist-barX}]
  We have already shown the assertion in the case where $\HH$ is the space of hypersurfaces, as above.
  Thus we assume that $r \geq 2$. Without loss of generality, we can assume that $e_1 \geq p$.
  Let us take $Z \subset \PP^{N-1}$ to be an $(N-r)$-dimensional complete intersection variety defined by
  $r-1$ homogeneous polynomials
  $f^2, \dots, f^r \in K[z_1, \dots, z_N]$ of degrees $e^2, \dots, e^r$,
  such that $0 < \#(\Sing (Z)) < \infty$ (for example, a cone over a smooth variety).
  By Bertini's theorem,
  we can choose a hypersurface $Y^1 \subset \PP^{N-1}$ defined by a homogeneous polynomial
  $f^1 \in K[z_1, \dots, z_N]$ of degree $e^1$
  such that the intersection
  $Y := Y^1 \cap Z \subset \PP^{N-1}$ is smooth,
  where $Y^1 \cap \Sing (Z) = \emptyset$.
  Note that Bertini's theorem for subvarieties of $\PN$ in any $p \geq 0$
  is deduced from \cite[II, Theorem~8.18]{hartshorne} by using $d$-uple embeddings.

  Let $\beta = (\beta_1: \dots: \beta_N) \in \Sing (Z)$.
  Then we have
  \[
  f^1(\beta) \neq 0
  \textand
  \rk D((f^2, \dots, f^r), \beta; \PP^{N-1}) < r-1.
  \]
  Here, without loss of generality, we can assume that $\beta_1 \neq 0$.
  We regard $Y$ as a subvariety of $\PN$ contained in the hyperplane $(z_0 = 0) \simeq \PP^{N-1}$.

  Now, let us define a complete intersection variety
  \[
  X_0 := (g^1= f^2 =  \dots = f^r = 0) \subset \PN,
  \]
  where
  \[
  g^1 := f^1 - \frac{f^1(\beta)}{\beta_1^{e^1-p}} \cdot z_0^{p} z_1^{e^1-p}
  \in K[z_0,z_1, \dots, z_N].
  \]
  Then $X_0$ is strange for $v$ because of $g^1_{z_0} = 0$.
  In addition,
  we have $\dim (\Sing (X_0)) < 1$, as follows:
  By definition of $g^1$, the intersection $X_0 \cap (z_0 = 0)$ is equal to the smooth variety $Y$.
  Then $X_0$ is smooth at any $x \in Y$.
  (This deduced from \cite[Theorem.~14.2]{Ma}, as follows:
  Let $\ell$ be a homogeneous linear polynomial satisfying $\ell(x) \neq 0$.
  Then, since $\sO_{Y, x} = \sO_{X_0,x}/ (z_0/\ell)$ is a regular local ring,
  the $r+1$ functions $z_0/\ell, g^1/\ell, f^2/\ell, \dots, f^r/\ell$ give a subset of
  a regular system of parameters of $\sO_{\PN,x}$.
  Hence $\sO_{X_0,x} = \sO_{\PN,x}/(g^1/\ell, f^2/\ell, \dots, f^r/\ell)$ is also regular.)
  Thus $\Sing (X_0)$ does not intersect with $(z_0 = 0)$,
  which implies that $\Sing (X_0)$ must be of dimension $< 1$.

  Next, we set
  $\alpha := (1: \beta_1: \dots: \beta_N) \in \PN$,
  where $\alpha \neq v$. 
  We have $\alpha \in \Sing (X_0)$, as follows:
  Since $\beta \in Z$, we have $f^2(\alpha) = \dots = f^r(\alpha) = 0$.
  By definition of $g^1$, it follows that $g^1(\alpha) = 0$. Hence $\alpha \in X_0$.
  Since $\rk D((f^2, \dots, f^r), \beta; \PP^{N-1}) < r-1$,
  we have
  $\rk D((g^1, f^2, \dots, f^r), \alpha; \PN) < r$.
  Hence $\alpha \in \Sing (X_0)$.
\end{proof}

\subsection{Irreducible components of the incidence scheme}
\label{sec:dimens-incid-vari}

Let us consider
\[
I = \Set*{((f^k), a) \in \Hv \times \PN}{
  \begin{aligned}
    &f^1(a) = \dots = f^r(a) = 0 \\
    &\text{and } \rk D ((f^k), a)  < r
  \end{aligned}
},
\]
the incidence scheme whose general members parameterize pairs $(X, a)$ such that
$X$ is a complete intersection variety defined by $(f^k) \in \Hv$
and that $a$ is a singular point of $X$.
Let $\prH: I \rightarrow \Hv$ and $\prP: I \rightarrow \PN$ be the first and second projections.
Here, the image $\prH(I) \subset \Hv$ is the parameter space of singular complete intersections.

For a member $(f^1, \dots, f^r) \in \Hv$, we set
\begin{equation*}  D' ((f^k), a) :=
  \begin{bmatrix}
    f^1_{z_1} (a) & \cdots & f^1_{z_{N-1}} (a)
    \\
    \vdots && \vdots
    \\
    f^r_{z_1} (a) & \cdots & f^r_{z_{N-1}} (a)
  \end{bmatrix}
\end{equation*}
(i.e., we remove $f^k_{z_N}$'s from $D ((f^k), a)$ defined in \ref{eq:defn-Dfa}),
and set
\[
I' := \Set*{((f^k), a) \in \Hv \times \PN}{
  \begin{aligned}
    & f^1(a) = \dots = f^r(a) = 0
    \\
    &\text{and }  \rk D' ((f^k), a) < r
  \end{aligned}
},
\]
where we have $I \subset I'$.
We will investigate irreducible components of $I'$ and $I$.

\begin{lem}\label{thm:dimI'}
  Let $\Lambda'$ be an irreducible component of $I'$.
  If $\prP(\Lambda') = \PN$,
  then $\Lambda'$ is of dimension $\geq \dim(H^v)$.
\end{lem}
\begin{rem}\label{thm:Er-codim}
  Let $M_{k} = M_{k}(r, n)$ be the set of ${(x_{i,j})_{1 \leq i \leq r, 1 \leq j \leq n} \in {\A^{rn}}}$
  such that
  \[
  \rk
  \begin{bmatrix}
    x_{1,1} & \dots & x_{1,n}
    \\
    \vdots && \vdots
    \\
    x_{r,1} & \dots & x_{r,n}
  \end{bmatrix}
  \leq k.
  \]
  Then $M_{k}$ is an irreducible subvariety of codimension $(r-k)(n-k)$ in ${\A^{rn}}$
  (see \cite[p. 67, II, \textsection 2, Prop.]{ACGH}).

\end{rem}

\begin{rem}\label{thm:def-Phie}
  Let $a = (a_0: \dots: a_N) \in \PN$ be a point with $a_N \neq 0$.
  Let $\hat{\Hv_e} \subset H^0(\PN, \sO(e))$ be the affine subvariety
  corresponding to $\Hv_e \subset \HH_e$. We consider the morphism
  \[
  \Phi_e: \hat{\Hv_e} \rightarrow \A^{N-1}: f \mapsto (f_{z_1}(a), \dots, f_{z_{N-1}}(a)).
  \]
  Then $\Phi_e$ is surjective.
  This is because, for $b = (b_1, \dots, b_{N-1}) \in \A^{N-1}$,
  we have 
  \[
  \Phi_e(b_1 / a_N^{e-1} \cdot z_1z_N^{e-1} + \dots + b_{N-1} / a_N^{e-1} \cdot z_{N-1}z_N^{e-1}) = b.
  \]
\end{rem}

\begin{proof}[Proof of \autoref{thm:dimI'}]
  Let $a = (a_0: \dots: a_N) \in \PN$ be general. Then we can assume $a_N \neq 0$.
  Let us  calculate the codimension of the intersection $I' \cap (\Hv \times \set{a})$,
  which is equal to the fiber of $\prP: I' \rightarrow \PN$ at $a$.
  First, we set
  \[
  \Phi = \bigoplus_{\RNk} \Phi_{e^k}: \hat\Hv \rightarrow \A^{r(N-1)}: (f^k)
  \mapsto D'((f^k), a),
  \]
  where $\hat\Hv := \bigoplus \hat{\Hv_{e_k}}\subset \bigoplus H^0(\PN, \sO(e^k))$ is the affine subvariety
  corresponding to $\Hv \subset \HH$.
  From \autoref{thm:def-Phie}, $\Phi$ is surjective.
  Moreover, $\Phi$ is a smooth morphism, since it is regarded as a linear map of vector spaces.
  Hence
  $\Phi^{-1}(M_{r-1}) \subset \hat\Hv$ is of codimension $N-r$,
  where $M_{r-1} = M_{r-1}(r, N-1) \subset \A^{r(N-1)}$
  is the subvariety defined in \autoref{thm:Er-codim}.
  Then each irreducible component of
  \begin{equation}\label{eq:fib-at-a}
    \Phi^{-1}(M_{r-1}) \cap F_a \subset \hat\Hv
  \end{equation}
  is of codimension $\leq N$,
  where $F_a := \set*{(f^k) \in \hat\Hv}{f^1(a) = \dots = f^r(a) = 0}$.
  Now, the projective variety in $\Hv$ corresponding to the affine variety \ref{eq:fib-at-a}
  can be identified with 
  \[
  I' \cap (\Hv \times \set{a}) \subset (\Hv \times \set{a}).
  \]
  Hence each irreducible component of $I' \cap (\Hv \times \set{a})$ is of codimension $\leq N$.

  Since $\Lambda' \rightarrow \PN$ is surjective and since $a \in \PN$ is general,
  there exists an irreducible component $W$ of
  $I' \cap (\Hv \times \set{a})$
  such that
  $W$ coincides with an irreducible component of
  $\Lambda' \cap (\Hv \times \set{a})$.
  The reason is as follows.
  Let $U = \Lambda' \setminus \bigcup_{i=1}^m V_i$,
  where $V_1, \dots, V_m$ are the irreducible components of $I'$ with $V_i \neq \Lambda'$.
  Since $U \rightarrow \PN$ is dominant,
  we can take an irreducible component $W$ of $I' \cap (\Hv \times \set{a})$
  such that $W \cap U \neq \emptyset$. Then $W$ is contained in the fiber
  $\Lambda' \cap (\Hv \times \set{a})$,
  and is indeed an irreducible component of this fiber.


  In addition, since $a$ is general,
  we have
  \[
  \codim(\Lambda', \Hv \times \PN) = \codim(W, \Hv \times \set{a})
  \]
  (see \cite[II, Ex. 3.22]{hartshorne}).
  Hence $\codim(\Lambda', \Hv \times \PN) \leq N$, which implies the assertion.
\end{proof}

\begin{lem}\label{thm:notin-zN0}
  Let $\Lambda' \subset I'$ be an irreducible subset.
  If $\prP(\Lambda') \not\subset (z_N = 0)$,
  then $\Lambda' \subset I$.
\end{lem}
\begin{proof}
  Let $((f^k), a) \in \Lambda'$ be a general member.
  Since $\prP(\Lambda') \not\subset (z_N = 0)$,
  $a = (a_0: \dots: a_N)$ satisfies
  $a_N \neq 0$. 
  Then, since $f^k_{z_0} = 0$, it follows from Euler's formula $\sum_{j=0}^{N} a_j f^k_{z_j} (a) = e^k f^k(a) = 0$ that
  \[
  f^k_{z_N} (a) = 
  - (a_1/a_N \cdot f^k_{z_1} (a) + \cdots + a_{N-1}/a_N \cdot f^k_{z_{N-1}} (a))
  \]
  holds for $\RNk$.
  Thus the last column vector of the matrix $D((f^k), a)$
  defined in \ref{eq:defn-Dfa} can be written as
  a linear combination of the other column vectors.
  Then, since $D'((f^k), a)$ is of rank $< r$, so is $D((f^k), a)$.
  It follows that $((f^k), a) \in I$.
\end{proof}

\begin{thm}\label{thm:sing-str2}
  Let $v \in \PN$ be a point, and let $e^1, \dots, e^r$ be $r$ integers greater than $1$.
  Let $\Hv = \Hv_{e^1} \times \dots \times \Hv_{e^r}$ be
  the parameter space defined in
  \ref{eq:defHv}.
  Assume that the space $\HH$ defined in \ref{eq:defHH} is \emph{not} equal to
  the parameter space of quadric hypersurfaces of odd dimensions
  in $p=2$.
  Then $\prH(I) = \Hv$, which means that
  every complete intersection variety defined by a member of $\Hv$ is singular.
\end{thm}

\begin{proof}
  If $p=0$ or $e^k < p$ for any $k$, then it follows from
  \autoref{thm:cone-strange} and \autoref{thm:e_k--p}
  that every $X$ defined by a member of $\Hv$ is a cone with vertex $v$,
  in particular, is singular. 

  Now we assume that $p > 0$ and $e^k \geq p$ for some $k$.
  From \autoref{thm:exist-barX}, we can take $((f^k), \alpha) \in \Hv \times \PN$
  such that $\alpha \neq v$ is an isolated singular point of
  the complete intersection variety $X_0 \subset \PN$ defined by $(f^k)$.
  
  We denote by $\PGL(\PN; v) \subset \PGL(\PN)$
  the group of automorphisms $\sigma$ of $\PN$ such that $\sigma(v) = v$.
  Let us consider
  the subset of $I$ parameterizing pairs $(\sigma(X_0), \sigma(\alpha))$ with $\sigma \in \PGL(\PN; v)$,
  which is actually given by the orbit of $((f^k), \alpha)$ in $I$,
  \[
  \set*{(((\sigma^{-1})^*f^k), \sigma(\alpha)) \in I}{\sigma \in \PGL(\PN; v)},
  \]
  where $(\sigma^{-1})^*f^k(z) := f^k(\sigma^{-1}(z))$.
  Now we take an irreducible component $\Lambda$ of $I$ containing the above orbit.
  Then $\prP(\Lambda) = \PN$.
  Here we have
  \begin{equation}\label{eq:dim-pr2-V}
    \dim (\Lambda) = \dim (\prH(\Lambda)).
  \end{equation}
  The reason is as follows:
  The fiber of $\prH: I \rightarrow \Hv$ at $(f^k)$
  is equal to $\Sing (X_0)$.
  Since $\alpha$ is isolated, the set $\set{\alpha}$
  is an irreducible component of $\Sing(X_0)$.
  Since each irreducible component of a fiber of
  $\Lambda \rightarrow \prH(\Lambda)$ must be of dimension
  $\geq \dim(\Lambda) - \dim(\prH(\Lambda))$, the equality $\dim(\Lambda) = \dim(\prH(\Lambda))$ holds.

  Let $\Lambda' \subset I'$ be an irreducible component of $I'$ such that
  $\Lambda \subset \Lambda'$.
  From \autoref{thm:dimI'},
  $\dim (\Lambda') \geq \dim (\Hv)$.
  From \autoref{thm:notin-zN0},  we have $\Lambda = \Lambda'$.
  From \ref{eq:dim-pr2-V}, we have $\dim(\prH(\Lambda)) \geq \dim(\Hv)$,
  and hence $\prH(\Lambda) = \Hv$.
\end{proof}

\begin{proof}[Proof of \autoref{thm:sing-str}]
  The equivalence ``(a) $ \Leftrightarrow $ (c)'' follows immediately from \autoref{thm:ci-cohom-cond} and \autoref{thm:rem-dim-1}.
  The implication ``(b) $ \Rightarrow $ (a)'' follows as in \autoref{thm:quad}.
  Now we show the implication ``(a) $ \Rightarrow $ (b)''.
  Let $X \subset \PN$ be a smooth $(e^1, \dots, e^r)$-complete intersection variety,
  and assume that $X$ is strange for a point $v \in \PN$.
  It is sufficient to consider the case where $X$ is non-degenerate.
  Then $e^k > 1$ for any $k$.
  As in \autoref{thm:defpolyX-in-H}, it follows from \autoref{thm:fk_z0=0} that
  $X$ is defined by a member of $(f^k) \in \Hv$.
  Since $X$ is smooth, it follows from \autoref{thm:sing-str2}
  that $X$ must be a quadric of odd dimension in $p=2$.
\end{proof}

\vspace{1em}
\paragraph{Acknowledgments}

The author would like to thank Hajime~Kaji for helpful discussions and suggestions.
The author also would like to thank Satoru~Fukasawa and Yasunari~Nagai
for their valuable comments.
The author was partially supported by JSPS KAKENHI Grant Number 25800030.

\vspace{1ex}
\end{document}